%% file: bpaffine.tex
\documentclass[12pt]{amsart}
\usepackage{amsmath,amsthm,amsfonts,amssymb,latexsym,verbatim,bbm,longtable}
\input xy
\xyoption{all}
\usepackage{hyperref,multirow}
\usepackage[shortlabels]{enumitem}
\usepackage{young}
\usepackage{tikz, ifthen}
\usepackage[nomessages]{fp}
\usepackage{mathtools}
\usetikzlibrary{calc}
\usetikzlibrary{chains}
\usepackage{booktabs,caption}

\input{circularstaircase.tex}

\allowdisplaybreaks[2] \textwidth15.1cm \textheight22cm \headheight12pt \oddsidemargin.4cm
\theoremstyle{plain}
\newtheorem{theorem}{Theorem}[section]
\newtheorem{thm}[theorem]{Theorem}

\newtheorem{proposition}[theorem]{Proposition}
\newtheorem{prop}[theorem]{Proposition}
\newtheorem{example}[theorem]{Example}
\newtheorem{definition}[theorem]{Definition}
\newtheorem{defn}[theorem]{Definition}

\newtheorem{cor}[theorem]{Corollary}

\newcommand{\arr}{\rightarrow}

\newcommand{\Z}{\mathbb{Z}}

\newcommand{\mcB}{\mathcal{B}}
\newcommand{\mcC}{\mathcal{C}}
\newcommand{\mcD}{\mathcal{D}}
\newcommand{\mcE}{\mathcal{E}}
\newcommand{\mcF}{\mathcal{F}}

\DeclareMathOperator{\adj}{adj}

\DeclareMathOperator{\flip}{flip}

\newcommand\arint[2]{\overrightarrow{[s_{#1},s_{#2}]}}

\newcommand{\mlabel}[1]{\(s_{\mathrlap{#1}}\)}
\newcommand{\dnode}[2][chj]{\node[#1,label={below:\mlabel{#2}}] {};}
\newcommand{\dydots}{ \node[chj,draw=none,inner sep=1pt] {\dots};}

\tikzset{node distance=2em, ch/.style={circle,draw,on chain,inner sep=2pt},chj/.style={ch,join}, line width=1pt,baseline=-1ex}

\newcounter{x}
\newcounter{y}
\newcounter{z}

\newcommand*\cubecolors[1]{%
  \ifcase#1\relax
  \or\colorlet{cubecolor}{cyan}%
  \or\colorlet{cubecolor}{green}%
  \or\colorlet{cubecolor}{yellow}%
  \or\colorlet{cubecolor}{pink}%
  \or\colorlet{cubecolor}{purple}%
  \or\colorlet{cubecolor}{blue}%
  \else
    \colorlet{cubecolor}{white}%
  \fi
}
\newcommand\yaxis{180}
\newcommand\zaxis{-27}
\newcommand\xaxis{90}

\newcommand\topside[3]{
  \fill[fill=cubecolor, draw=black,shift={(\xaxis:#1)},shift={(\yaxis:#2)},
  shift={(\zaxis:#3)}] (0,0) -- (1,0) -- (0.5,0.25) --(-0.5,0.25)--(0,0);
}

\newcommand\leftside[3]{
  \fill[fill=cubecolor, draw=black,shift={(\xaxis:#1)},shift={(\yaxis:#2)},
  shift={(\zaxis:#3)}] (0,0) -- (0,-1) -- (-0.5,-0.75) --(-0.5,0.25)--(0,0);
}

\newcommand\rightside[3]{
  \fill[fill=cubecolor, draw=black,shift={(\xaxis:#1)},shift={(\yaxis:#2)},
  shift={(\zaxis:#3)}] (0,0) -- (1,0) -- (1,-1) --(0,-1)--(0,0);
}

\newcommand\topsidedashed[3]{
  \fill[fill=cubecolor, draw=black,shift={(\xaxis:#1)},shift={(\yaxis:#2)},
  shift={(\zaxis:#3)}] [dashed](0,0) -- (1,0) -- (0.5,0.25) --(-0.5,0.25)--(0,0);
}

\newcommand\leftsidedashed[3]{
  \fill[fill=cubecolor, draw=black,shift={(\xaxis:#1)},shift={(\yaxis:#2)},
  shift={(\zaxis:#3)}] [dashed](0,0) -- (0,-1) -- (-0.5,-0.75) --(-0.5,0.25)--(0,0);
}

\newcommand\rightsidedashed[3]{
  \fill[fill=cubecolor, draw=black,shift={(\xaxis:#1)},shift={(\yaxis:#2)},
  shift={(\zaxis:#3)}] [dashed](0,0) -- (1,0) -- (1,-1) --(0,-1) -- (0,0);
}

\newcommand\cube[3]{
  \topside{#1}{#2}{#3} \leftside{#1}{#2}{#3} \rightside{#1}{#2}{#3}
}

\newcommand\cubedashed[3]{
  \topsidedashed{#1}{#2}{#3} \leftsidedashed{#1}{#2}{#3} \rightsidedashed{#1}{#2}{#3}
}

\newcommand\ppAff[2]{
 \setcounter{x}{0}
 \foreach \a in {#2} {
    \addtocounter{x}{1}
    \setcounter{y}{-1}
    \foreach \b in \a {
      \addtocounter{y}{1}
      \ifthenelse{\b=0}{\addtocounter{y}{0}}{
        \FPeval{\result}{clip(#1-\the\value{y}-1)}
        \cubecolors{\b}
        \ifthenelse{\result=0\OR\value{y}=0}
        {
        \cubedashed{\value{x}}{\value{y}}{0};
        \FPeval{\result}{clip(#1-\the\value{y}-1)}
        \draw[draw=black,shift={(\xaxis:\value{x})},shift={(\yaxis:\value{y})},
  shift={(\zaxis:0)}] (0.5,-0.5) node {\textsf{0}};}
        {
        \cube{\value{x}}{\value{y}}{0};
        \FPeval{\result}{clip(#1-\the\value{y}-1)}
        \draw[draw=black,shift={(\xaxis:\value{x})},shift={(\yaxis:\value{y})},
  shift={(\zaxis:0)}] (0.5,-0.5) node {\textsf{\result}};}}
    }
  }
}

\newcommand\ppFinite[2]{
 \setcounter{x}{0}
 \foreach \a in {#2} {
    \addtocounter{x}{1}
    \setcounter{y}{-1}
    \cubecolors{\value{x}}
    \foreach \b in \a {
    \cubecolors{\b}
      \addtocounter{y}{1}
      \setcounter{z}{-1}
      \foreach \c in {0,...,\b} {
        \addtocounter{z}{1}
      \ifthenelse{\c=0}{\setcounter{z}{-1},\addtocounter{y}{0}}{
        \FPeval{\newz}{clip(0.55*\the\value{z})}
        \cube{\value{x}}{\value{y}}{0};
        \FPeval{\result}{clip(#1-\the\value{y})}
        \ifthenelse{\result<0}{
            \FPeval{\resultt}{clip(10+\result)}
            \draw[draw=black,shift={(\xaxis:\value{x})},shift={(\yaxis:\value{y})},
            shift={(\zaxis:0)}] (0.5,-0.5) node {\textsf{\resultt}};}{
        \draw[draw=black,shift={(\xaxis:\value{x})},shift={(\yaxis:\value{y})},
  shift={(\zaxis:0)}] (0.5,-0.5) node {\textsf{\result}};}}
      }
    }
  }
}

\begin{document}

\title{Smooth Schubert varieties in the affine flag variety of type $\tilde{A}$}

\author{Edward Richmond}
\email{edward.richmond@okstate.edu}

\author{William Slofstra}
\email{weslofst@uwaterloo.ca}

\begin{abstract}
    We show that every smooth Schubert variety of affine type $\tilde{A}$ is
    an iterated fibre bundle of Grassmannians, extending an analogous result by
    Ryan and Wolper for Schubert varieties of finite type $A$. As a
    consequence, we finish a conjecture of Billey-Crites that a Schubert
    variety in affine type $\tilde{A}$ is smooth if and only if the
    corresponding affine permutation avoids the patterns $4231$ and $3412$.
    Using this iterated fibre bundle structure, we compute the generating
    function for the number of smooth Schubert varieties of affine type
    $\tilde{A}$.
\end{abstract}
\maketitle

\section{Introduction}

Let $X$ be a Kac-Moody flag variety, and let $W$ be the associated Weyl group.
Although $X$ can be infinite-dimensional, it is stratified by
finite-dimensional Schubert varieties $X(w)$, where $w \in W$. It is natural to
ask when $X(w)$ is smooth or rationally smooth, and this question is
well-studied \cite{BL00}. For the finite-type flag variety of type $A_n$, the Weyl
group is the permutation group $\mathfrak{S}_n$,\footnote{Note that we use
$\mathfrak{S}$ to refer to groups, and $S$ to refer to sets of simple
reflections.} and the Lakshmibai-Sandhya theorem states that $X(w)$ is smooth
if and only if $w$ avoids the permutation patterns $3412$ and $4231$
\cite{LS90}. From another angle, the Ryan-Wolper theorem states that $X(w)$ is
smooth if and only if $X(w)$ is an iterated fibre bundle of Grassmannians of
type $A$ \cite{Ry87,Wo89}.  Haiman used the Ryan-Wolper theorem to enumerate
smooth Schubert varieties \cite{Ha92,Bo98}. The Lakshmibai-Sandhya theorem, the
Ryan-Wolper theorem, and the enumeration of smooth and rationally smooth
Schubert varieties has been extended to all finite types (see \cite{Bi98,BP05},
\cite{RS14}, and \cite{RS15} respectively). The latter enumeration uses a data
structure called \emph{staircase diagrams}, which keeps track of iterated fibre
bundle structures.

There are also characterizations of smoothness and rational smoothness that
apply to all Kac-Moody types \cite{Ca94,Ku96}. For instance, a theorem of
Carrell and Peterson states that $X(w)$ is rationally smooth if and only if the
Poincare polynomial $P_w(q)$ of $X(w)$ is palindromic, meaning that the
coefficients read the same from top-degree to bottom-degree and vice-versa
\cite{Ca94}.  However, much less is known about the structure of (rationally)
smooth Schubert varieties in general Kac-Moody type. The one exception is
affine type $\tilde{A}$, where Billey and Crites have characterized the
elements $w$ for which $X(w)$ is rationally smooth \cite{BC12}.  In this case, the Weyl group $W$ is the affine permutation group $\tilde{\mathfrak{S}}_n$.  As part of
their characterization, they prove that if $X(w)$ is smooth, then $w$ must
avoid the \emph{affine permutation patterns} $3412$ and $4231$. They conjecture
the converse, that $X(w)$ is smooth if $w$ avoids these two patterns.

The purpose of this paper is to extend what we know about finite-type Schubert
varieties to affine type $\tilde{A}$. For smooth Schubert varieties, we show:
\begin{thm}\label{T:main1}
    Let $X(w)$ be a Schubert variety in the full flag variety of type $\tilde{A}$.
    Then the following are equivalent:
    \begin{enumerate}[(a)]
        \item\label{T:main11} $X(w)$ is smooth.
        \item\label{T:main12} $w$ avoids the affine permutation patterns $3412$ and $4231$.
        \item\label{T:main13} $X(w)$ is an iterated fibre bundle of Grassmannians of finite type $A$.
    \end{enumerate}
\end{thm}
In particular, this finishes the proof of Billey and Crites' conjecture. We
note that the proof relies heavily on ideas from both \cite{BC12} and
\cite{RS14}. One corollary (explained in Section \ref{S:staircase}) is that
there is a bijection between smooth Schubert varieties in the full flag variety
of type $\tilde{A}_n$, and spherical staircase diagrams over the Dynkin diagram
of type $\tilde{A}_n$.  This allows us to enumerate smooth Schubert varieties
in affine type $\tilde{A}_n$:
\begin{thm}\label{T:enum}
    Let $A(t) = \sum a_n\, t^n$, where $a_n$ is the number of smooth
    Schubert varieties in the full flag variety of type $\tilde{A}_n$. Then
    \begin{equation*}
        A(t) = \frac{P(t) - Q(t) \sqrt{1-4t}}{(1-t)(1-4t)\left(1-6t+8t^2-4t^3\right)} 
    \end{equation*}
    where
    \begin{equation*}
        P(t) = (1-4t)\left(2-11t+18t^2-16t^3+10t^4-4t^5\right) 
    \end{equation*}
    and
    \begin{equation*}
        Q(t) = (1-t)(2-t)\left(1-6t+6t^2\right). 
    \end{equation*}
\end{thm}

In Table \ref{TBL:count}, we list the number of smooth Schubert varieties of
type $\tilde A_n$, or equivalently, the number of affine permutations in
$\tilde{\mathfrak{S}}_n$ which avoid $3412$ and $4231$ for $n\leq 9$.

\begin{table}[h]
    \begin{tabular}{cccccccc}
        \toprule
         $n=2$ & $n=3$ & $n=4$ & $n=5$ & $n=6$&$n=7$ &$n=8$ &$n=9$  \\
        \midrule
         5     & 31     & 173    &  891 & 4373 & 20833 &97333 &448663  \\

        \midrule
    \end{tabular}
    \caption{Number of smooth Schubert varieties of type $\tilde A_n$. }
    \label{TBL:count}
\end{table}
Using the generating series, we can also determine the asymptotics of $a_n$.  Let
\begin{equation*}
    \alpha:=\frac{1}{6}\left(4-\sqrt[3]{17+3\sqrt{33}}+\sqrt[3]{-17+3\sqrt{33}}\right)\approx 0.228155
\end{equation*}
which is the unique real root of the polynomial $1-6t+8t^2-4t^3$ from the denominator of the generating function $A(t)$.
\begin{cor}\label{C:enum} Asymptotically, we have $a_n \sim \alpha^{-n}$.
\end{cor}
\begin{proof}
    The singularity of $A(t)$ with smallest modulus is the root $\alpha$ of
    the polynomial $1-t6+8t^2-4t^3$. Since this occurs with multiplicity one,
    \cite[Theorem IV.7]{FS09} states that $a_n \sim C / \alpha^{n+1}$, where
    $C:=\lim_{t\rightarrow \alpha} A(t)(\alpha-t).$  In this case, $C = \alpha$
    (at the moment we do not have an explanation for this interesting
    coincidence).
\end{proof}

In finite type $A$, every rationally smooth Schubert variety is smooth. In
affine type $\tilde{A}$, this is not true \cite{BM10, BC12}. For the full flag
variety, Billey and Crites show that there is just one infinite
family of rationally smooth Schubert varieties which are not smooth.
\begin{thm}[Theorem 1.1, Corollary 1.2 and Remark 2.16 of \cite{BC12}]\label{T:BC1}
    A Schubert variety $X(w)$ in the full flag variety of type $\tilde{A}_n$ is
    rationally smooth if and only if either
    \begin{enumerate}[(a)]
        \item $w$ avoids the affine permutation patterns $3412$ and $4231$, or
        \item $w$ is a twisted spiral permutation (in which case, $X(w)$ is not smooth).
    \end{enumerate}
\end{thm}
The twisted spiral permutations are defined in the next section. As we will
explain, it is easy to see that if $w$ is a twisted spiral permutation, then
$X(w)$ is a fibre bundle over a rationally smooth Grassmannian Schubert
variety, with fibre equal to the full flag variety of type $A_{n-1}$. The base
of this fibre bundle is a \emph{spiral Schubert variety}, a family of
Schubert varieties in the affine Grassmannian introduced by Mitchell
\cite{Mi86}. Thus it follows from Theorems \ref{T:main1} and \ref{T:BC1} that,
just as for finite-type Schubert varieties, a Schubert variety in the full flag
variety of type $\tilde{A}_n$ is rationally smooth if and only if it is an
iterated fibre bundle of rationally smooth Grassmannian Schubert varieties.
As we explain in the next section, this also holds for Schubert varieties in the
partial flag varieties of affine type $\tilde{A}_n$. Finally, we note that
Theorem \ref{T:main1} was first proved in a preprint version of \cite{RS14},
but was removed during the publication process. Here we give a variant of the
original proof, along with an additional proof using staircase diagrams.

The rest of the paper is organized as follows. In the next section, we give the
first proof of Theorem \ref{T:main1}, along with the related results for
partial flag varieties. In Section \ref{S:staircase}, we review the notion of a
staircase diagram, and give the second proof of Theorem \ref{T:main1}. Finally,
in Section \ref{S:enumeration} we prove Theorem \ref{T:enum}.

\subsection{Acknowledgements}

We thank Erik Slivken for suggesting the bijection between increasing staircase
diagrams and Dyck paths given in the proof of Proposition \ref{P:increasing}.
We thank Sara Billey for helpful discussions.

\section{BP decompositions in affine type $\tilde{A}$}\label{S:affine_typeA}

As in the introduction, let $W$ denote the Weyl group of a Kac-Moody group $G$. Let
$S$ be the set of simple reflections of $W$, and let $\ell:W\rightarrow \Z_{\geq 0}$ be the length
function.  A parabolic subgroup of $W$ is a subgroup $W_J$ generated by
a subset $J \subseteq S$.  Every (left) $W_J$-coset has a unique minimal-length
element, and the set of minimal-length coset representatives is denoted by
$W^J$ (similarly, the set of minimal length right coset representatives is
denoted by ${}^J W$. The partial flag variety $X^J$ is stratified by Schubert
varieties $X^J(w)$ for $w \in W^J$, each of complex dimension $\ell(w)$. The
Poincar\'{e} polynomial of $w \in W^J$ is
\begin{equation*}
    P_w^J(q) = \sum_{x \leq w \text{ and } x \in W^J} q^{\ell(x)},
\end{equation*}
where $\leq$ is Bruhat order. As mentioned in the introduction, a theorem of
Carrell and Peterson states that $X^J(w)$ is rationally smooth if and only if
$P^J_w(q)$ is palindromic, meaning that $q^{\ell(w)} P^J_w(q^{-1}) = P^J_w(q)$
\cite{Ca94}. In the case that $J=\emptyset$, let $X(w) := X^{\emptyset}(w)$ and $P_w(q) := P^{\emptyset}_w(q)$.

Given $J \subseteq K \subseteq S$, every element $w \in W^J$ can be written
uniquely as $w = vu$ where $v \in W^K$ and $u \in W_K^J := W_K \cap W^J$.  This
is called the parabolic decomposition of $w$ with respect to $K$. A parabolic
decomposition $w = vu$ is a \emph{Billey-Postnikov (BP) decomposition (relative to
$J$)} if
$P^J_w(q) = P^K_v(q) \cdot P_u^J(q)$. There are other equivalent combinatorial
characterizations which are computationally easy to check. In particular, if $J
= \emptyset$, then $w=vu$ is a BP decomposition if and only if $S(v) \cap K
\subseteq D_L(u)$, where $S(w) := \{s \in S : s \leq w\}$ is the support set of
$w \in W$, and $D_L(w) := \{s \in S : sw \leq w\}$ is the left descent set of
$w$ (see \cite[Proposition 4.2]{RS14} for more details).

The following result from \cite{RS14} gives a geometric interpretation for BP
decompositions.
\begin{thm}\emph{(\cite[Theorem 3.3]{RS14})}\label{T:bp1}
    Given $J \subseteq K \subset S$, let $w = vu$, where $w \in W^J$, $v \in
    W^K$, and $u \in W_K^J$. Then the following are equivalent:
    \begin{enumerate}[(a)]
        \item The decomposition $w = vu$ is a BP decomposition.
        \item The natural projection $\pi: X^J(w) \arr X^K(v)$ is
            Zariski-locally trivial with fibre $X^J(u)$.
    \end{enumerate}
    Consequently, if $w = vu$ is a BP decomposition then $X^J(w)$ is
    (rationally) smooth if and only if $X^J(u)$ and $X^K(v)$ are (rationally)
    smooth.
\end{thm}
A Grassmannian BP decomposition is a BP decomposition $w = vu$ with respect to
a set $K$ with $|S(w) \setminus K| =1$. The main technical result of \cite{BC12} is:
\begin{prop}[\cite{BC12}]\label{P:BC2}
    Let $W$ be the Weyl group of type $\tilde{A}_n$. If $w\in W$, as an affine permutation avoids $3412$ and
    $4231$ then either $w$ or $w^{-1}$ has a Grassmannian BP decomposition
    $vu$, where both $v$ and $u$ belong to proper parabolic subgroups
    of $\tilde{A}_n$.
\end{prop}
Proposition \ref{P:BC2} is proved implicitly in \cite{BC12}; in particular, see
the proof of Theorem 3.1, and the discussion before Corollary 7.1 in \cite{BC12}.
The following extension of Proposition \ref{P:BC2} shows that we don't need to
look at $w^{-1}$ to find a BP decomposition:
\begin{prop}\label{P:affine_onesided}
    Let $W$ be a Weyl group of type $\tilde{A}_n$. If $w\in W$ avoids both
    $3412$ and $4231$, then $w$ has a Grassmannian BP decomposition $w=vu$
    where both $v$ and $u$ belong to proper parabolic subgroups of
    $\tilde{A}_n$.

    Furthermore, one of the following is true:
    \begin{enumerate}[(a)]
    \item $w$ is the maximal element of a parabolic proper subgroup of $\tilde{A}_n$ or,
    \item $w=vu$ is a BP decomposition with respect to $S(w) \setminus \{s\}$, for some $s
    \not\in D_R(w)$.
    \end{enumerate}
\end{prop}
The proof is similar to that of \cite[Theorem 6.1]{RS14}; for the convenience
of the reader, we give a complete proof for the $\tilde{A}_n$ case.
\begin{proof}
    If $S(w)$ is a proper subset of $S$, then $W_{S(w)}$ is finite of type $A$,
    and the proposition is exactly Theorem 6.1 of \cite{RS14}. Hence, we
    assume that $S(w) = S$ throughout.  By Proposition \ref{P:BC2}, it suffices to prove that if $w^{-1}$ has a Grassmannian BP decomposition with both factors belonging to proper parabolic subgroups, then $w$ has a Grassmannian BP decomposition with respect to some $K = S \setminus \{s\}$ where $s\notin D_R(w)$.  Assume that $w^{-1}$ has a Grassmannian BP decomposition with both factors belonging to proper parabolic subgroups.  Hence there is a subset $K = S
    \setminus \{s\}$ such that $w = uv$ with $u \in W_K$, $v \in {}^K W$, and
    $S(v) \cap K \subseteq D_R(u)$, the right descent set of $u$. Furthermore,
    $S(v)$ is a proper subset of $S = S(w)$; consequently, $S(u) = K$ and
    $K \setminus S(v)$ is non-empty.

    Next, we claim that $u$ has a Grassmannian BP decomposition $u = v'u'$ with
    respect to some $K' = S \setminus \{s'\}$, where $s' \not\in S(v)$. Indeed,
    $u$ is rationally smooth of type $A$. If $u$ is the maximal element of
    $W_K$ then we can take $s'$ to be any element of $K \setminus S(v)$. If $u$
    is not maximal, then $u$ has a Grassmannian BP decomposition $u = v' u'$
    with respect to $K' = K \setminus \{s'\}$, where $s' \not\in D_R(u)$
    \cite[Theorem 6.1]{RS14}, and hence $s' \not\in S(v)$.

    Since $s' \not\in S(v)$, we have $w = v' (u' v)$ is the parabolic decomposition
    of $w$ with respect to $K'$. Since $v'u'$ is a BP decomposition, and $u' v$
    is reduced, we have
    $$S(v')\cap K' \subseteq D_L(u') \subseteq D_L(u' v),$$
    and thus $w = v' (u' v)$ is a BP decomposition. Finally $S(v') \subseteq K \subsetneq
    S$, completing the proof of the first part of the proposition.

    We now show we can choose $s'\notin D_R(w)$.  First, if $u'$ is not maximal, then choose a
    Grassmannian BP decomposition with respect to $s' \not\in D_R(u)$. Since
    $s' \not\in S(v)$, we get that $s' \not\in D_R(w)$ as well
    (this follows, for instance, from the fact that if $a\neq b \in D_R(w)$,
    then $b \in D_R(wa)$). If $u'$ is instead the maximal element, then we can
    take $s' \in K \setminus S(v)$ such that $s'$ is adjacent (in the Dynkin
    diagram) to some element of $S(v)$. It follows from \cite[Lemma 6.4]{RS14}
    that $s' \not\in D_R(w)$.
\end{proof}
\begin{cor}\label{C:affine_onesided}
    Suppose $w \in W^J$, where $W$ is the Weyl group of type $\tilde{A}_n$.
    If $X^J(w)$ is a smooth Schubert variety and $S(w) \setminus J \neq
    \emptyset$, then $w$ has a Grassmannian BP decomposition with respect to
    some $J \subseteq K \subsetneq S(w)$, in which each factor belongs to a
    proper parabolic subgroup of $W$.
\end{cor}
\begin{proof}
    Let $u_0$ be the maximal element of $W_{J \cap S(w)}$.
    Then $w' := wu_0$ is a BP decomposition for $w'$ with respect to $J$.  By
    Theorem \ref{T:bp1}, $X(w')$ is a fibre bundle over $X^J(w)$ with fibre
    $X(u_0)$. Since $X(u_0)$ is a full flag variety of type $A$ (and hence is smooth), it follows that
    $X(w')$ is smooth. By Theorem \ref{T:BC1}, $w'$ must avoid $3412$ and
    $4231$. We claim that $w'$ has a Grassmannian BP decomposition $w' = vu'$
    with respect to some $K = S(w) \setminus \{s\}$ such that $s \not\in J$.
    Indeed, if $w'$ is the maximal element of $W_{S(w')}$, then $w'$ has a
    Grassmannian BP decomposition with respect to any $s \in S(w')$, including
    any element of $S(w) \setminus J$. If $w'$ is not maximal, then by
    Proposition \ref{P:affine_onesided}, $w'$ has a Grassmannian BP decomposition
    where $s \not\in D_R(w')$, and since $D_R(w')$ contains $S(u_0) = J \cap S(w)$, we
    conclude that $s \not\in J$.

    Since $s \not\in J$, we conclude that $K = S \setminus \{s\}$ contains $J$,
    and thus that $u' = u u_0$ for some $u \in W_{K}^J$. It is easy to see that
    $w = v u$ is a BP decomposition of $w$ with respect to $K$, as desired.
\end{proof}

When combined with Theorem \ref{T:bp1}, Corollary \ref{C:affine_onesided}
implies that every smooth element in type $\tilde{A}$ has a complete BP
decomposition in the following sense:
\begin{defn}[\cite{RS15}]\label{D:complete}
    Let $W$ be a Coxeter group. A \emph{complete BP decomposition of $w \in W^J$
    (with respect to $J$)} is a factorization $w = v_1 \cdots v_m$, where
    $m = |S(w) \setminus J|$, and, if we let $u_i := v_i \cdots v_m \in W^J$,
    then $u_i = v_i u_{i+1}$ is a Grassmannian BP decomposition with respect to
    $K_i := S(u_{i+1}) \cup J$ for all $1 \leq i < m$.

    A complete \emph{maximal} BP decomposition $w = v_1 \cdots v_m$ is a complete
    BP decomposition $w = v_1 \cdots v_m$ as above such that $v_i$ is maximal
    in $W_{K_{i-1} \cap S(v_i)}^{S(v_i)}$ for all $1  \leq i \leq m$.
\end{defn}

Hence (and similarly to \cite[Corollary 3.7]{RS14}), Theorem \ref{T:bp1} and
Corollary \ref{C:affine_onesided} imply that a Schubert variety $X^J(w)$ of
type $\tilde{A}_n$ is smooth if and only if it is an iterated fibre bundle of
Grassmannians.
\begin{cor}\label{C:fibrebundle}
    A Schubert variety $X^J(w)$ in a partial flag variety of type $\tilde{A}_n$
    is smooth if and only if there is a sequence
    \begin{equation*}
        X^J(w) = X_m \arr X_{m-1} \arr \cdots \arr X_{1} \arr X_{0} = \text{pt},
    \end{equation*}
    where each map $X_{i} \arr X_{i-1}$ is a Zariski-locally trivial fibre bundle
    whose fibre is a Grassmannian variety of type $A$.
\end{cor}
\begin{proof}
    A Zariski-locally trivial fibre bundle with a smooth base and fibre is
    itself smooth, so it follows that if $X^J(w)$ is an iterated fibre bundle
    in the above sense, then $X^J(w)$ is smooth.

    For the converse, let $W$ be the Weyl group of type $\tilde{A}_n$, and
    suppose that $X^J(w)$ is smooth, where $w \in W^J$. If $S(w) \subseteq J$,
    then $X^J(w)$ is a point, and the theorem is trivial. If $|S(w) \setminus J| =
    m \geq 1$, then by Corollary \ref{C:affine_onesided}, $w$ has a complete
    BP decomposition $w = v_1 \cdots v_m$ in which $S(v_i)$ is a strict subset
    of $S$. Let $u_i := v_{i} \cdots v_m \in W^J$, and let $w_i := v_1 \cdots v_i$,
    so that $u_1 = w_m = w$ and $u_{m+1} = w_0 = e$. In addition, set
    $K_i := S(u_{i+1}) \cup J$, so that
    \begin{itemize}
        \item $J = K_m \subsetneq K_{m-1} \subsetneq \cdots \subsetneq K_1 \subsetneq K_0 := S(w)$,
        \item $|K_{i-1} \setminus K_i|=1$,
        \item $v_i \in W_{K_{i-1}}^{K_i}$,
        \item $u_i = v_i u_{i+1}$ is a BP decomposition with respect to $K_i$ (and relative to $J$).
    \end{itemize}
    By Theorem \ref{T:bp1}, $X^{K_i}(v_i)$ is smooth for all $1 \leq i \leq m$.
    Since $S(v_i)$ is a strict subset of $S$, it follows that $X^{K_i}(v_i)$ is
    a smooth Grassmannian Schubert variety of finite type $A$. But it is
    well-known (see for instance \cite{BL00}) that the only smooth Grassmannian
    Schubert varieties of finite type $A$ are themselves Grassmannians
    (specifically, $v_i$ must be maximal in $W_{S(v_i)}^{K_i \cap S(v_i)}$,
    so $w = v_1 \cdots v_m$ is a complete maximal BP decomposition).
    By \cite[Lemma 4.3]{RS14}, $w_i = w_{i-1} v_i$ is a BP decomposition with respect
    to $K_{i-1}$ (and relative to $K_i$), so if we set $X_i = X^{K_i}(w_i)$,
    then by Theorem \ref{T:bp1} the standard projection $X_i \arr X_{i-1}$ is a
    Zariski-locally trivial fibre bundle with fibre $X^{K_i}(v_i)$.
\end{proof}

\begin{proof}[Proof of Theorem \ref{T:main1}]
    If $X(w)$ is smooth, then $w$ avoids $3412$ and $4231$ by Theorem
    \ref{T:BC1}, so part (a) implies part (b).

    Suppose $w$ avoids $3412$ and $4231$. If $w = vu$ is a
    parabolic decomposition, then $u$ also avoids $3412$ and $4231$ \cite[Lemma
    3.10]{BC12}. Thus Proposition \ref{P:affine_onesided} implies that $w$ has a complete BP decomposition, in
    which every factor belongs to a proper parabolic subgroup of $W$. Every
    rationally smooth Grassmannian Schubert variety in finite type $A$ is
    smooth by the Carrell-Peterson theorem \cite{CK03}, and hence a Grassmannian.
    Thus the proof of Corollary \ref{C:fibrebundle} implies that $X(w)$ is an
    iterated fibre bundle of Grassmannians.  Hence part (b) implies part (c).

    Similarly, if $X(w)$ is an iterated fibre bundle of Grassmannians, then
    $X(w)$ is smooth as in the proof of Corollary \ref{C:fibrebundle}, so part (c)
    implies part (a).
\end{proof}

We finish the section by looking at rationally smooth Schubert varieties,
starting with the twisted spiral permutations. Suppose $W$ has type
$\tilde{A}_n$, and let $S = \{s_0,\ldots,s_{n-1}\}$, where $s_i$ is the simple
reflection corresponding to node $i$ in the Dynkin diagram of $\tilde{A}_n$ as
shown in Figure \ref{F:dynkin1}.
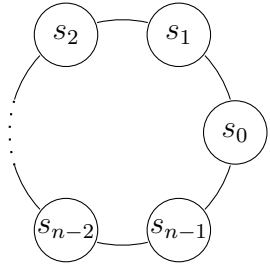
\begin{figure}

\begin{tikzpicture}
\def \radius {1.5cm}
\def \margin {17} 
\foreach \s in {1,...,6}
{ 
  \draw[-] ({60 * (\s - 1)+\margin}:\radius)
    arc ({60 * (\s - 1)+\margin}:{60 * (\s)-\margin}:\radius);
}
  \node[draw, circle, minimum size=24pt] at ({60*0}:\radius) {$s_{0}$};
  \node[draw, circle, minimum size=24pt] at ({60*1}:\radius) {$s_{1}$};
  \node[draw, circle, minimum size=24pt] at ({60*2}:\radius) {$s_{2}$};
  \node[draw, circle, minimum size=24pt] at ({60*4}:\radius) { };
  \node at ({60 * 4}:\radius) {$s_{n-2}$};
  \node[draw, circle, minimum size=24pt] at ({60*5}:\radius) {};
  \node at ({60 * 5}:\radius) {$s_{n-1}$};
  \draw[-, thick, loosely dotted] ({60*3 + \margin}:\radius) arc ({60*3 + \margin}:{60*3 - \margin}:\radius);
\end{tikzpicture}
    \caption{The Dynkin diagram of type $\tilde{A}_n$}
    \label{F:dynkin1}
\end{figure}
Given $0 \leq i < n$ and $k \geq 1$, define
\begin{equation*}
    x(i,k) = s_{i+k-1} s_{i+k-2} \cdots s_{i} \text{ and }
        y(i,k) = s_{i-k+1} s_{i-k+2} \cdots s_i,
\end{equation*}
where the indices of the $s_j$'s are interpreted modulo $n$. Both $x(i,k)$ and
$y(i,k)$ belong to $W^{S \setminus \{s_i\}}$. A \emph{spiral permutation} is an
element of $W$ of the form $x(i,k(n-1))$ or $y(i,k(n-1))$ for some $k \geq
2$.\footnote{Note that the length of these elements is a multiple of $n-1$,
even though the rank of $W$ is $n$. In particular, if $k=1$ then these elements
belong to a parabolic subgroup of finite type $A$, so we exclude this case.}
The spiral permutations were first studied by Mitchell \cite{Mi86}, who showed
that the corresponding Grassmannian Schubert varieties, called \emph{spiral
Schubert varieties}, are rationally smooth.

A \emph{twisted spiral permutation} is an element of $W$
of the form $w = vu$, where $v$ is a spiral permutation $x(i,k(n-1))$ or
$y(i,k(n-1))$, and $u$ is the maximal element of $W_{S \setminus \{s_i\}}$.
Note that $w = vu$ is a BP decomposition with respect to $J = S \setminus \{s_i\}$.
In addition, $D_R(w)$ contains $D_R(u)=S
\setminus \{s_i\}$, and since $\tilde{A}_n$ is infinite, $D_R(w)$ cannot be equal
to $S$, so $D_R(w)$ must be equal to $S \setminus \{s_i\}$. Thus we get
a version of Proposition \ref{P:affine_onesided} for all rationally smooth
elements of $W$: if $X(w)$ is rationally smooth, then either 
\begin{itemize}
    \item $w$ is the maximal element of a proper parabolic subgroup of $W$, or 
    \item $w$ has a Grassmannian BP decomposition with respect to $K = S(w) \setminus \{s\}$, where $s \not\in D_R(w)$. 
\end{itemize}
This means that we can repeat the proofs of Corollaries
\ref{C:affine_onesided} and \ref{C:fibrebundle} to get:
\begin{cor}\label{C:fibrebundle2}
    A Schubert variety $X^J(w)$ in a partial flag variety of type $\tilde{A}_n$
    is rationally smooth if and only if there is a sequence
    \begin{equation*}
        X^J(w) = X_m \arr X_{m-1} \arr \cdots \arr X_{1} \arr X_{0} = \text{point},
    \end{equation*}
    where each map $X_{i} \arr X_{i-1}$ is a Zariski-locally trivial
    fibre bundle whose fibre is a rationally smooth Grassmannian Schubert
    variety of type $A$ or $\tilde{A}$.
\end{cor}
It is implicit in the proof of Corollary \ref{C:fibrebundle2} that if $X^J(w)$
is rationally smooth, then we can construct such a sequence where all the
fibres are either Grassmannians of finite type $A$, or spiral Schubert
varieties.  In fact, these are the only rationally smooth Grassmannian Schubert
varieties, by a theorem of Billey and Mitchell \cite{BM10}.

In the finite-type analogue of Corollary \ref{C:fibrebundle2}, every rationally
smooth Grassmannian Schubert variety is almost-maximal \cite{RS14}. Although we
don't need that fact here, it is interesting to note that the spiral permutations
are also almost-maximal.

\section{Staircase diagrams for affine type $\tilde{A}$}\label{S:staircase}

The main fact we had to establish in the previous section was that every 3412-
and 4231-avoiding element of type $\tilde{A}$ has a complete maximal BP
decomposition. To prove this fact, we showed that the existence of BP
decompositions for $w$ or $w^{-1}$ implies the existence of BP decompositions
for $w$. The same proof strategy was used in \cite{RS14}. Both results are
instances of a more general result, which we formulate as follows:
\begin{prop}\label{P:universal}
    Let $\mcF$ be a family of Coxeter groups which is closed under parabolic
    subgroups (i.e. if $W \in \mcF$, then $W_J \in \mcF$ for every $J \subseteq
    S$). Let $\mcC$ be a class of elements of the groups of $\mcF$ such that
    if $w \in \mcC$, then
    \begin{enumerate}[(1)]
        \item $w^{-1} \in \mcC$, and
        \item $u \in \mcC$ for all parabolic decompositions $w = vu$.
        \item either $w$ or $w^{-1}$ has a maximal Grassmannian BP decomposition,
            i.e. a Grassmannian BP decomposition $vu$ where $v$ is maximal in
            $W_{S(v)}^{S(u) \cap S(v)}$.
    \end{enumerate}
    Then every $w \in \mcC$ has a complete maximal BP decomposition.
\end{prop}
The point of this section is to show that Proposition \ref{P:universal} has a
short proof using staircase diagrams. Taking $\mcF$ to be the Weyl groups of
finite type $A$ and affine type $\tilde{A}$, and $\mcC$ to be the class of
permutations avoiding the pattern $3412$ and $4231$, we get a proof of Theorem \ref{T:main1}
by a somewhat different route. Our proof of Proposition \ref{P:universal} will
still hold if we replace ``maximal'' by ``maximal or almost-maximal'' BP
decompositions, and hence Proposition \ref{P:universal} also gives an alternate
path to the results on existence of BP decompositions in \cite{RS14}. For
simplicity, we restrict ourselves to maximal BP decompositions.

%
%

As we will also use staircase diagrams in the next section, we briefly review
the definition from \cite{RS15}. Let $G$ be a graph with vertex set $S$, and
recall that a subset $B \subseteq S$ is connected if the subgraph of $G$
induced by $B$ is connected. If $\mcD$ is a collection of subsets of $S$ and
$s\in S$, we set
\begin{equation*}
    \mcD_s:=\{B\in\mcD \ |\ s\in B\}.
\end{equation*}
Staircase diagrams are then defined as follows:
\begin{definition}[\cite{RS15}]\label{D:staircase}
    Let $\Gamma$ be a graph with vertex set $S$.  Let $\mcD = (\mcD, \preceq)$ be a partially ordered subset of $2^S$ not
    containing the empty set. We say that $\mcD$ is a \emph{staircase diagram}
    if the following are true:
    \begin{enumerate}[(1)]
        \item Every $B\in\mcD$ is connected, and if $B$ covers $B'$ then $B\cup B'$ is connected.
        \item The subset $\mcD_s$ is a chain for every $s \in S$.
        \item If $s\adj t$, then $\mcD_s\cup \mcD_t$ is a chain, and $\mcD_s$ and $\mcD_t$
            are saturated subchains of $\mcD_s \cup \mcD_t$.
        \item If $B\in\mcD$, then there is some $s\in S$ (resp. $s'\in S$) such
            that $B$ is the minimum element of $\mcD_s$ (resp. maximum element of
            $\mcD_{s'}$).
    \end{enumerate}
\end{definition}
The definition is symmetric with respect to the partial order, so if $\mcD =
(\mcD,\preceq)$ is a staircase diagram, and $\preceq'$ is the reverse order to
$\preceq$, then $(\mcD, \preceq')$ is also a staircase diagram, called
$\flip(\mcD)$.  Finally, if $\mcD$ is a staircase diagram over the Coxeter graph of a Coxeter group $W$, we say $\mcD$ is \emph{spherical} if $W_B$ is a finite group for all $B\in\mcD$.
The main result about staircase diagrams is:
\begin{thm}[\cite{RS15}, Theorem 5.1, Theorem 3.7, and Corollary 6.4]\label{T:stairs}
    Let $\Gamma$ be the Coxeter graph of a Weyl group $W$. Then there is a bijection
    between spherical staircase diagrams over $\Gamma$, and elements of $w$ with a complete
    maximal BP decomposition (relative to $\emptyset$).

    Furthermore, if a staircase diagram $\mcD$ corresponds to $w \in W$, then
    $\flip(\mcD)$ corresponds to $w^{-1}$.
\end{thm}

\begin{proof}[Second proof of Proposition \ref{P:affine_onesided}]
    The proof is by induction on $|S(w)|$. Suppose that $w$ has a maximal
    Grassmannian BP decomposition $w = vu$. Since $|S(u)| < |S(v)|$, we can
    conclude by induction that $u$ has a complete maximal BP decomposition.
    But these means that $w$ has a complete maximal BP decomposition, and
    we are done. Similarly, if $w^{-1}$ has a maximal Grassmannian BP
    decomposition, then we conclude that $w^{-1}$ has a complete maximal
    BP decomposition. But this means that $w^{-1}$ comes from a staircase
    diagram $\mcD$, so $w$ corresponds to $\flip(\mcD)$. In particular, $w$
    must also have a complete maximal BP decomposition.
\end{proof}

\section{Enumeration of smooth Schubert varieties}\label{S:enumeration}

Let $\tilde\Gamma_n$ be the Coxeter graph of type $\tilde A_n$ with vertices
$\tilde S_n=\{s_0,s_1,\ldots,s_{n-1}\}$ as in Figure \ref{F:dynkin1}.  In this
case, proper, connected subsets of $\tilde S_n$ are simply intervals on the
cycle graph $\tilde\Gamma_n$.  Define the interval


\begin{equation*}
\arint{i}{j}:=\begin{cases}
                \{s_i,\ldots,s_j\} & \text{if $i\leq j$}\\
                \{s_i,\ldots,s_{n-1},s_0,\ldots,s_j\} & \text{if $i>j$}\\
                \end{cases}
\end{equation*}

We represent a staircase diagram $\mcD$ pictorially with a collection of ``blocks" where if $B_2$ lies above $B_1$ and $B_1\cup B_2$ is connected, then $B_1\prec B_2$ in $\mcD$.
For example, the staircase diagram
$$\mcD=\{\arint{0}{3}\prec\arint{7}{1}\prec\arint{5}{7}\prec\arint{3}{6}\}$$
over $\tilde\Gamma_{10}$ could be represented in Figure \ref{F:affine_circ_diagram}.

    \def\wedgeoffset{45}
    \def\wedgeinnerradius{3}
    \def\wedgeouterradius{4}
\begin{figure}[h]
    \begin{tikzpicture}[x={(0.866cm,0.5cm)},y={(0cm,1cm)},z={(0.866cm,-0.5cm)},scale=0.5]
        \upperbox[252,288,\wedgeinnerradius,\wedgeouterradius,0,1,green,3]
        \upperbox[288,324,\wedgeinnerradius,\wedgeouterradius,0,1,green,2]
        \slantupperbox[324,360,\wedgeinnerradius,\wedgeouterradius,0,1,-0.5,green,1]
        \slantupperbox[324,360,\wedgeinnerradius,\wedgeouterradius,1,2,-0.5,yellow,1]
        \upperbox[252,288,\wedgeinnerradius,\wedgeouterradius,1,2,cyan,3]
        \upperbox[216,252,\wedgeinnerradius,\wedgeouterradius,1,2,cyan,4]
        \upperbox[180,216,\wedgeinnerradius,\wedgeouterradius,0,1,pink,5]
        \upperbox[180,216,\wedgeinnerradius,\wedgeouterradius,1,2,cyan,5]
        \lowerbox[144,180,\wedgeinnerradius,\wedgeouterradius,0,1,pink,6]
        \lowerbox[144,180,\wedgeinnerradius,\wedgeouterradius,1,2,cyan,6]
        \lowerbox[108,144,\wedgeinnerradius,\wedgeouterradius,-1,0,yellow,7]
        \lowerbox[108,144,\wedgeinnerradius,\wedgeouterradius,0,1,pink,7]
        \slantlowerbox[0,36,\wedgeinnerradius,\wedgeouterradius,-0.5,0.5,-0.5,green,0]
        \slantlowerbox[0,36,\wedgeinnerradius,\wedgeouterradius,0.5,1.5,-0.5,yellow,0]
        \slantlowerbox[36,72,\wedgeinnerradius,\wedgeouterradius,0,1,-0.5,yellow,9]
        \slantlowerbox[72,108,\wedgeinnerradius,\wedgeouterradius,-0.5,0.5,-.5,yellow,8]
    \end{tikzpicture}
    \caption{Staircase diagram on $\tilde\Gamma_{10}$}
    \label{F:affine_circ_diagram}
\end{figure}
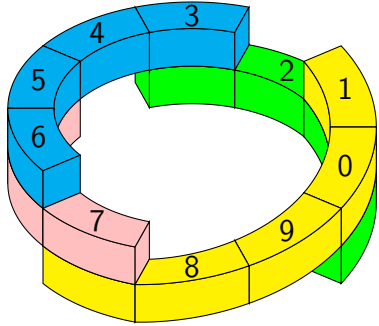

However, for the sake of convenience we will represent the cycle graph
$\tilde\Gamma_n$ as a line graph with vertex $s_0$ each end point, as follows.

$$
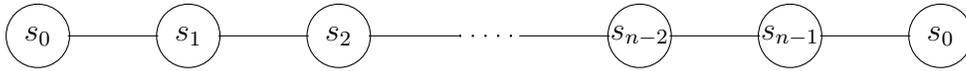
\begin{tikzpicture}
\def \n {6}
\def \radius {2.5cm}
\def \margin {0.4} 
\foreach \s in {1,...,\n}
{ 
  \draw ({2*\s-2+\margin} ,0) -- ({2*\s-\margin},0);
}
  \node[draw, circle, minimum size=24pt] at (0,0) {$s_{0}$};
  \node[draw, circle, minimum size=24pt] at (2,0) {$s_{1}$};
  \node[draw, circle, minimum size=24pt] at (4,0) {$s_{2}$};
  \node[draw, circle, minimum size=24pt] at (8,0) { };
  \node at (8,0) {$s_{n-2}$};
  \node[draw, circle, minimum size=24pt] at (10,0) {};
  \node at (10,0) {$s_{n-1}$};
  \node[draw, circle, minimum size=24pt] at (12,0) {$s_{0}$};
  \draw[-,thick, loosely dotted] (6-\margin,0)--(6+\margin,0);
\end{tikzpicture}
$$


We can then draw the staircase diagram in two-dimensions; for instance, the
staircase diagram in Figure \ref{F:affine_circ_diagram} is represented as:

$$
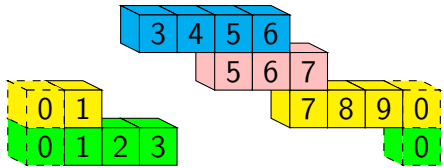
\begin{tikzpicture}[scale=0.5]
    \ppAff{11}{
        {2,0,0,0,0,0,0,2,2,2,2},
        {3,3,3,3,0,0,0,0,0,3,3},
        {0,0,0,4,4,4,0,0,0,0,0},
        {0,0,0,0,1,1,1,1,0,0,0}}
\end{tikzpicture}
$$

Note that if $s_0\in B\in\mcD$, then $B$ appears as a ``disconnected" block in
the pictorial representation of $\mcD$.

By Corollary \ref{C:fibrebundle} and Theorem \ref{T:stairs}, to enumerate
smooth Schubert varieties of type $\tilde A_n$, it suffices to enumerate
spherical staircase diagrams over the graph $\tilde\Gamma_n$.  Since every
proper subgraph of $\tilde\Gamma_n$ is a Dynkin diagram of finite type, the
only non-spherical staircase diagram is $\mcD=\{\tilde S_n\}$.

We first consider staircase diagrams over the Dynkin graph of finite type
$A_n$.  Let $\Gamma_n$ denote the line graph with vertex set
$S_n=\{s_1,\ldots,s_{n}\}$.  In this case, we denote interval blocks on line
graph $\Gamma_n$ as simply $[s_i,s_j]$ for $i<j$.

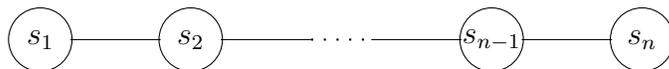
$$
\begin{tikzpicture}
\def \n {5}
\def \radius {2.5cm}
\def \margin {0.4} 
\foreach \s in {2,...,\n}
{ 
  \draw ({2*\s-2+\margin} ,0) -- ({2*\s-\margin},0);
}
  \node[draw, circle, minimum size=24pt] at (2,0) {$s_{1}$};
  \node[draw, circle, minimum size=24pt] at (4,0) {$s_{2}$};
  \node[draw, circle, minimum size=24pt] at (8,0) { };
  \node at (8,0) {$s_{n-1}$};
  \node[draw, circle, minimum size=24pt] at (10,0) {$s_{n}$};
  \draw[-,thick, loosely dotted] (6-\margin,0)--(6+\margin,0);
\end{tikzpicture}
$$

A staircase diagram $\mcD$ is said to have \emph{full support} if every vertex
of the underlying graph appears in some $B \in \mcD$.  We say a staircase
diagram $\mcD$ over $\Gamma_n$ of full support is \emph{increasing} if we can
write $$\mcD=\{B_1\prec B_2\prec \cdots\prec B_m\}$$ with $s_1\in B_1$.
Similarly, a staircase diagram is \emph{decreasing} if if we can write
$\mcD=\{B_1\succ B_2\succ \cdots\succ B_m\}$ with $s_1\in B_1$.  Let
$M^{\pm}(n)$ denote the set of fully supported increasing/decreasing staircase
diagrams over $\Gamma_n$.  Define the generating series
\begin{equation}
    A_M(t):=\sum_{n=1}^{\infty} m_n\ t^n
\end{equation}
where $m_n:=|M^+(n)|=|M^-(n)|$.
\begin{proposition}\label{P:increasing}
    The generating function $\displaystyle A_M(t)=\frac{1-2t-\sqrt{1-4t}}{2t}$.
\end{proposition}

\begin{proof}
    The proposition can be proved by modifying the proof of \cite[Proposition
    8.3]{RS15}. In this paper we present an alternate proof by giving a
    bijection between fully supported increasing staircase diagrams and Dyck
    paths.  Indeed, let $\mcD=\{B_1\prec B_2\prec \cdots\prec B_m\}$ be a fully
    supported increasing staircase diagram on $\Gamma_n$.  For each
    $B_i\in\mcD,$ define the numbers
    \begin{equation*}
        r(B_i):=\#\{s\in B_i\setminus B_{i-1}\}\qquad\text{and}\qquad
            u(B_i):=\#\{s\in B_i\setminus B_{i+1}\}
    \end{equation*}
    where we set $B_0=B_{m+1}=\emptyset.$
    Let $P(\mcD)$ denote the lattice path in $\Z^2$ from $(0,0)$ to $(n,n)$ which
    takes $r(B_1)$ steps to the right, then $u(B_1)$ steps going up, followed by
    $r(B_2)$ steps to the right, then $u(B_2)$ steps going up and so forth (See
    Example \ref{Ex:Dyckpath_bijection}).  Since $\mcD$ is fully supported, we have
    that
    \begin{equation*}
        \sum_{i=1}^{m} r(B_i)=\sum_{i=1}^{m} u(B_i)=n
    \end{equation*}
    and hence $P(\mcD)$ terminates at $(n,n)$.  Definition \ref{D:staircase}
    implies $r(B_i),u(B_i)>0$ and that the partial sums
    \begin{equation*}
        \sum_{k=1}^{i} r(B_k)\geq \sum_{k=1}^{i} u(B_k)
    \end{equation*} for all $i\leq m.$  Thus
    $P(\mcD)$ is a Dyck path.  Conversely, any Dyck path is given by a sequence of
    positive pairs $(r_i,u_i)$ giving steps to the right followed up steps going
    up.  Set $u_0:=0$ and define
    \begin{equation*}
        \bar B_i:=\left\{s_j\in S\ |\  \sum_{k=1}^{i-1} u_k<j\leq \sum_{k=1}^{i} r_k\right\}
    \end{equation*}
    and $\bar\mcD:=\{\bar B_1\prec \bar B_2\prec \cdots\prec \bar B_m\}$.  It is
    easy to see that $\bar\mcD$ is a fully supported staircase diagram and that
    this construction is simply the inverse of the map $P$.  The proposition now
    follows from the generating function for Dyck paths which is given by Catalan
    numbers.
\end{proof}

\begin{example}\label{Ex:Dyckpath_bijection}
    Consider the staircase diagram $\mcD=({s_1}\prec[s_2,s_5]\prec[s_4,s_6])$ on $\Gamma_6$.  The sequence of pairs $(r_i,u_i)$ is $((1,1),(4,2),(1,3))$ and corresponding Dyck path $P(\mcD)$ is given below.
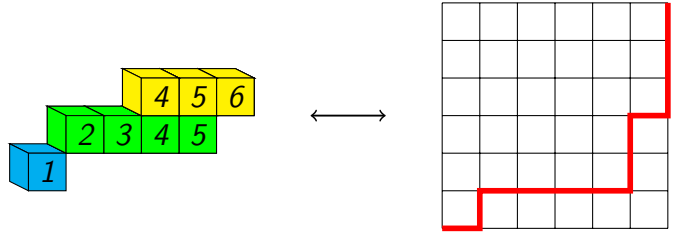
\begin{equation*}
    \begin{tikzpicture}[scale=0.5]
        \ppFinite{6}{
            {0,0,0,0,0,1},
            {0,2,2,2,2},
            {3,3,3,0}}
        \draw[<->, shift={(2.5,2)},black,line width=.25mm](0,0)--(2,0);
        \draw[shift={(6,-1)},step=1.0,black] (0,0) grid (6,6);
        \draw[shift={(6,-1)},red, line width=.75mm] (0,0)--(1,0)--(1,1)--(5,1)--(5,3)--(6,3)--(6,6);
    \end{tikzpicture}
\end{equation*}
\end{example}

The idea behind enumerating staircase diagrams over $\tilde\Gamma_n$ is to
partition a staircase diagram into a disjoint union of increasing and
decreasing staircase diagrams of finite type $A$.  To do this precisely, we
introduce the notion of a broken staircase diagram.  We say that a partially
ordered collection of subsets $(\mcB,\prec)$ of vertices of the graph $\Gamma_n$ is a
\emph{broken staircase diagram} if
\begin{equation*}
    \mcB=\{B\cap S_n\ |\ B\in\mcD\}
\end{equation*}
for some $\mcD\in M^+(n+1)\cup M^-(n+1)$ where the partial order on $\mcB$ is
induced from $\mcD$.  Note that broken staircase diagrams are allowed to violate
part (4) of Definition \ref{D:staircase}, and must be either increasing or
decreasing. In particular, if
$\mcB=\{B_1\prec\cdots\prec B_m\}$ is broken, then it may be possible for $B_{m}\subset
B_{m-1}$.  Define the generating series
\begin{equation*}
    A_B(t)=\sum_{n=1}^{\infty} b_n\, t^n
\end{equation*}
where $b_n$ denotes the number of increasing (or equivalently, decreasing)
broken staircase diagrams on $\Gamma_n$.

\begin{proposition}\label{P:broken_increasing}
    The generating function $\displaystyle A_B(t)=\frac{(1-t)A_M(t)}{t}-1.$
\end{proposition}

\begin{proof}
Clearly $b_0=0$ and $b_1=1$, so we assume that $n\geq 2$.  Let
$\mcD=\{B_1\prec\cdots \prec B_m\}\in M^+(n+1)$ and let $\mcB(\mcD)=\{B\cap
S_n\ |\ B\in\mcD\}$ denote the corresponding broken staircase diagram.  If
block index $k\leq m-2$, then $B_k\subseteq S_n$.  Hence $\mcB(\mcD)$
determines $\mcD$ up to the last two blocks $B_{m-1}, \, B_m$.  If $B_m\subset
B_{m-1}$, then $\mcB(\mcD)$ uniquely determines $\mcD$ as shown in Figure
\ref{F:broken_case1}.
\begin{figure}[h]
    \begin{tikzpicture}[scale=0.5]
        \begin{scope}[shift={(-2,0)}]
        \ppFinite{5}{
                {0,0,0,3,3},
                {0,2,2,2,0},
                {0,1,1,0}}
        \end{scope}
        \draw[->, shift={(0,1.5)},black,line width=.25mm](0,0)--(1,0);
        \begin{scope}[shift={(7,0)}]
        \ppFinite{5}{
                {0,0,0,3,3},
                {0,2,2,2,0},
                {1,1,1,0}}
        \draw[dashed, red,line width=0.25mm] (0,0)--(0,4);
        \end{scope}
        \end{tikzpicture}
        \caption{The broken staircase diagram $\mcB(\mcD)$ determining $\mcD$.}
        \label{F:broken_case1}
\end{figure}
If $B_m \not\subset B_{m-1}$, then there are two possibilities for $\mcD$ given
$\mcB(\mcD)$.  Either $\mcD_{s_n}=\{B_m\}$ and thus
$s_n,s_{n+1}\in B_m$, or $\mcD_{s_n}=\{B_{m-1}\}$ which implies
$B_m=\{s_{n+1}\}$ (see Figure \ref{F:broken_case2}).
\begin{figure}[h]
    \begin{tikzpicture}[scale=0.5]
        \begin{scope}[shift={(-2,0)}]
        \ppFinite{5}{
                {0,0,0,3,3},
                {0,0,2,2,0},
                {0,1,1,0}}
        \end{scope}
        \draw[->, shift={(0,1.5)},black,line width=.25mm](0,0)--(1,0);
        \begin{scope}[shift={(7,0)}]
        \ppFinite{5}{
                {0,0,0,3,3},
                {0,0,2,2,0},
                {1,1,1,0}}
        \draw[dashed, red,line width=0.25mm] (0,1)--(0,4);
        \end{scope}
        \draw (10,1.5) node {\text{or}};
        \begin{scope}[shift={(16,0)}]
        \ppFinite{5}{
                {0,0,0,3,3},
                {0,0,2,2,0},
                {0,1,1,0,0},
                {4,0}}
        \draw[dashed, red,line width=0.25mm] (0,1)--(0,5);
        \end{scope}
        \end{tikzpicture}
        \caption{Two possibilities for $\mcD$ given $\mcB(\mcD)$.}
        \label{F:broken_case2}
\end{figure}
In the latter case, removing the last block from $\mcD$ gives a unique
staircase diagram in $M^+(n)$. Hence $b_n = m_{n+1}-m_n$ and
\begin{equation*}
    t+tA_B(t)=A_M(t)-tA_M(t).
\end{equation*}
This proves the proposition.
\end{proof}

We can now state the main bijection:
\begin{prop}\label{P:affineSD_structure}
    There is a bijection between fully-supported spherical staircase diagrams on
    $\tilde{\Gamma}_n$, and pairs $[(\mcB_1,\ldots,\mcB_{2k}),v]$, where
    \begin{itemize}
        \item $\mcB_i$ is a broken staircase on $\Gamma_{n_i}$, $n_i \geq 1$,
        \item $\sum_{i=1}^{2k} n_i = n$,
        \item for all $1 \leq  i \leq 2k-1$, if $\mcB_i$ is increasing (resp. decreasing)
            then $\mcB_{i+1}$ is decreasing (resp. increasing), and
        \item $v$ is a distinguished vertex in $\mcB_{2k}$.
    \end{itemize}
\end{prop}
\begin{proof}
    For the purpose of this proof, we let $s_{j+nk} = s_j$ for any $k$ and
    $0 \leq j < n$.  Suppose $\mcD$ is a fully-supported staircase diagram on
    $\tilde{\Gamma}_n$, and $B$ is any block of $\mcD$, say $B = \arint{i}{j}$.
    By Definition \ref{D:staircase}, part (3), all the blocks of
    $\mcD_{s_{j+1}}$ are comparable with $B$. If $B$ has an upper cover $B'
    \succeq B$ in $\mcD_{s_j} \cup \mcD_{s_{j+1}}$, then there are no elements of
    $\mcD_{s_{j+1}}$ below $B$, since then $\mcD_{s_{j+1}}$
    would not be saturated in $\mcD_{s_j} \cup \mcD_{s_{j+1}}$. And vice-versa,
    if $B$ has a lower cover in $\mcD_{s_j} \cup \mcD_{s_{j+1}}$ then there are
    no elements of $\mcD_{s_{j+1}}$ above $B$. Consequently we can say that $B$
    has a unique cover $B'$ containing $s_{j+1}$. We call $B'$ the right cover
    of $B$.

    Choose some block $B_1$, and let $B_1,\ldots,B_m$ be a sequence where
    $B_{i+1}$ is the right cover of $B_i$ for $1  \leq i < m$, and $B_1$ is the
    right cover of $B_{m}$. Then every block of $\mcD$ must appear in this
    sequence. Indeed, every vertex of $\tilde{\Gamma}_n$ appears in some block
    in this sequence, so every block of $\mcD$ is comparable to some
    element of the sequence. It follows that if there is a block of $\mcD$ not in
    the sequence, then there is a block $B$ not in the sequence which has an
    upper or lower cover $B'$ in the sequence. Then either $B$ will be the right
    cover of $B'$, or $B'$ will be the right cover of $B$. But the same argument
    as above shows that $B'$ has a unique left cover, and this is the only element
    with $B'$ as a right cover. So in both cases, $B$ must also be in the sequence,
    a contradiction.

    Let $B_{i_1},\ldots,B_{i_{m}}$ denote the subsequence of extremal
    blocks, i.e. blocks which are maximal or minimal. Note that if $B_{i_j}$ is
    maximal then $B_{i_{j+1}}$ must be minimal, and vice-versa. Since
    $\tilde{\Gamma}_n$ is a cycle, the same must apply to $B_{i_{m}}$ and
    $B_{i_1}$, and in particular $m$ must be even. By Definition \ref{D:staircase}, part (4),
    every extremal block contains a vertex which does not belong to any block.
    Let $1 \leq c_j \leq n$ be the index of the leftmost such vertex in
    $B_{i_j}$. By cyclically shifting the indices,
    we can assume that $1 \leq c_1 < c_2 < \ldots < c_m < n$.
    Finally, set
    \begin{equation*}
        J_{j} = \begin{cases} \arint{c_j}{c_{j+1}-1} & 1 \leq j < m \vspace{.1in}\\
                              \arint{c_{m}}{c_1-1} & j = m
                \end{cases},
    \end{equation*}
    so that $J_1,\ldots,J_m$ partitions $\tilde{S}_n$, and let
    \begin{equation*}
        \mcB_j = \{ B \cap J_j \ :\  B \in \mcD \text{ and } B \cap J_j \neq \emptyset \}
    \end{equation*}
    with the induced partial order. Since $B_{i_j}$ is the only block
    containing $s_{c_j}$, and no block of $\mcD$ contains any other \cite[Lemma
    2.6(b)]{RS15}, the block $B_{i_j}$ can meet at most two of the intervals
    $J_{k}$. Hence $\mcB_j$ is either an increasing or decreasing chain. It
    follows that $\mcB_1,\ldots,\mcB_m$ is a sequence of broken staircases as
    required. We set $v$ to be the vertex $s_{n-1}$, which is always in $\mcB_m$
    by construction.

    This construction gives a map from staircase diagrams to sequences of
    broken staircases with a marked vertex. To show that this map has an inverse,
    suppose that $\mcB = \{ B_1 \prec \cdots \prec B_m\}$ is an increasing broken staircase.
    If $B_m \subset B_{m-1}$ then we can think of $\mcB$ as the staircase diagram
    $\{ B_1 \prec \cdots \prec B_{m-1}\}$ with an additional broken block $B_m$ on top
    of $B_{m-1}$. If $B_m \not\subset B_{m-1}$, so $\mcB$ is a staircase diagram
    in its own right, then we think of $\mcB$ as a staircase diagram with an
    empty broken block on top of the block $B_m$, starting and ending after the
    rightmost vertex of $B_m$. If $\mcB'$ is then a decreasing staircase diagram,
    we can glue $\mcB$ and $\mcB'$ together by attaching the broken block of $\mcB$
    to the first block of $\mcB'$. We can similarly glue a decreasing broken
    staircase to an increasing broken staircase. Given a sequence $\mcB_1,\ldots,\mcB_{2k}$
    of alternately increasing and decreasing broken staircases, we can glue them
    together in order, and then glue $\mcB_{2k}$ to $\mcB_1$ to get a staircase
    diagram on a cycle. Labelling the vertices of the cycle with $s_0,\ldots,s_{n-1}$
    starting to the right of the marked vertex $v$, we get a staircase diagram on
    $\tilde{\Gamma}_n$, and this process inverts the above map.
\end{proof}

\begin{example}
The staircase diagram $$\mcD=\{\arint{1}{3}\prec\arint{3}{4}\prec\arint{5}{6}\succ\arint{6}{7}\succ\arint{7}{8}\prec\arint{9}{1}\}$$ has four extremal blocks and partitions into an alternating sequence of increasing and decreasing broken staircase diagrams as follows:
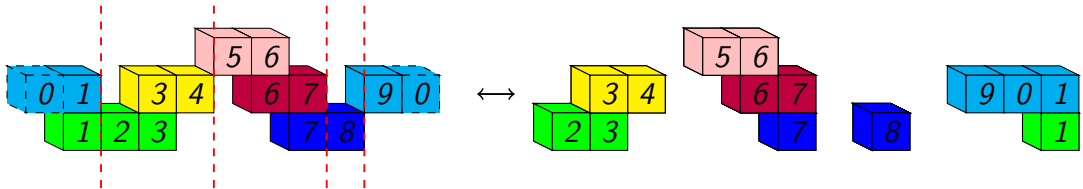
\begin{equation*}
    \begin{tikzpicture}[scale=0.5]
        \begin{scope}[shift={(-2,0)}]
        \ppAff{11}{
                {0,0,6,6,0,0,0,2,2,2,0},
                {1,1,0,5,5,0,3,3,0,1,1},
                {0,0,0,0,4,4,0}}
        \end{scope}
        \draw[dashed, red,line width=0.25mm] (-10,-1)--(-10,4);
        \draw[dashed, red,line width=0.25mm] (-7,-1)--(-7,4);
        \draw[dashed, red,line width=0.25mm] (-4,-1)--(-4,4);
        \draw[dashed, red,line width=0.25mm] (-3,-1)--(-3,4);
        \draw[<->, shift={(0,1.5)},black,line width=.25mm](0,0)--(1,0);
        \begin{scope}[shift={(6,0)}]
        \ppFinite{6}{
                {0,0,0,2,2},
                {0,0,3,3,0},
                {0,0,0,0}}
        \end{scope}
        \begin{scope}[shift={(9,0)}]
        \ppFinite{8}{
                {0,6,0,0},
                {0,5,5,0},
                {0,0,4,4}}
        \end{scope}
        \begin{scope}[shift={(10.5,0)}]
        \ppFinite{8}{
                {6}}
        \end{scope}
        \begin{scope}[shift={(16,0)}]
        \ppFinite{2}{
                {0,2,0,0},
                {0,1,1,1}}
        \end{scope}
        \end{tikzpicture}
\end{equation*}
\end{example}

Define the generating series
    $$\bar A(t)=\sum_{n=1}^{\infty} \bar a_n\, t^n$$
where $\bar a_n$ denotes the number of fully supported spherical staircase diagrams on $\tilde\Gamma_n.$

\begin{cor}\label{C:affine_full_support}
    The generating function $\displaystyle\bar A(t)=\frac{2 A_B(t) \cdot t\frac{d}{dt} A_B(t)}{1-A_B(t)^2}.$
\end{cor}
\begin{proof}
    Follows immediately from Proposition \ref{P:affineSD_structure}, and the
    fact that $t \frac{d}{dt} A_B(t)$ is the generating series for broken
    staircases with a marked vertex. Note that we get a factor of two because
    the first broken staircase can be increasing for decreasing.
\end{proof}
If a staircase diagram on $\tilde\Gamma_n$ is not fully supported, then it is a
disjoint union of fully supported staircase diagrams over a collection of
subpaths of the cycle.  Let $f_n$ denote the number of fully supported
staircase diagrams on the path $\Gamma_n$ and define the generating series
$$A_F(t):=\sum_{n=0}^{\infty} f_n\, t^n.$$ The following proposition is proved in
\cite[Proposition 8.3]{RS15}. We give an alternate proof using broken
staircase diagrams.
\begin{prop}\label{P:finite_full_support}\emph{(\cite[Proposition 8.3]{RS15})}
    The generating function $\displaystyle A_F(t)=\frac{A_M(t)}{1-A_B(t)}$.
\end{prop}
\begin{proof}
    We emulate the proof of Proposition \ref{P:affineSD_structure} as follows:
    Given a staircase diagram $\mcD$ on $\Gamma_{n}$, let $B_1,\ldots,B_k$ be
    the maximal and minimal blocks in order from left to right. Let $1 \leq m
    \leq k$ be the largest index such that $B_m$ is minimal in $\mcD$ (so actually,
    $m\in \{k-1,k\}$). For every $1  \leq j \leq m$, let $a_j$ be the index of
    the leftmost element of $B_j$ which is not contained in any other block, and
    let
    \begin{equation*}
        b_j = \begin{cases}
            a_{j+1}-1 & j < m \\
            n & j = m
        \end{cases}.
    \end{equation*}
    Let $\mcB_i = [s_{a_i},s_{b_i}]$. Then $\mcB_i$ is a broken staircase for
    $1 \leq i \leq m-1$, while $\mcB_m$ is an increasing staircase. This also
    implies that $\mcB_{m-1}$ is decreasing, $\mcB_{m-2}$ is increasing, and
    so on. It is not hard to see that this map is a bijection, and hence
    every staircase diagram over $\Gamma_n$ decomposes into a sequence of
    broken staircases, followed by an increasing staircase.
\end{proof}
One subtlety of the above bijection is that it seems to miss the case when
$\mcD$ is decreasing, or more generally, ends with a decreasing staircase.
However, a decreasing staircase decomposes into a decreasing broken staircase
followed by a single block. Since a single block is an increasing staircase,
the bijection will in fact count decreasing staircases correctly. Since
single blocks are both increasing and decreasing, the seemingly more
straightforward approach of allowing $\mcB_m$ to be increasing or decreasing
will lead to overcounts. The reason this problem doesn't arise in Proposition
\ref{P:affineSD_structure} is that in that bijection every $\mcB_i$ is
a broken staircase. We can always tell whether a broken staircase is increasing
or decreasing based on where it is glued to the adjacent broken staircase.

Finally, let $a_n$ denote the number of spherical staircase diagrams over the
graph $\tilde\Gamma_n$ and define
\begin{equation*}
    A(t)=\sum_{n=1}^{\infty} a_n\, t^n.
\end{equation*}
\begin{prop}\label{P:total_gen_series}
    The generating series $$A(t)=\bar A(t)+\frac{t\frac{d}{dt}\left(A_*(t)\right)}{1-A_*(t)}+\frac{t^2}{1-t}$$
    where $\displaystyle A_*(t)=\frac{tA_F(t)}{1-t}.$
\end{prop}
\begin{proof}
    First note that the generating function fully supported staircase diagrams over
    $\tilde\Gamma_n$ is $\bar A(t)$.  If a staircase diagram is not fully supported
    and nonempty, then it partitions into a sequence $(\mcD_0,\mcE_0,\ldots,
    \mcD_r,\mcE_r)$ where each $\mcD_k$ is a nonempty, fully supported staircase
    diagram over a path and $\mcE_k$ is an empty staircase diagram over a path of
    at least length one.  Moreover, we can choose such a partition such that $s_0$
    is in the support of $\mcD_0$ or $\mcE_0$. Thus we get a bijection between
    non-empty non-fully-supported staircase diagrams on $\tilde \Gamma_n$ and
    sequences $(\mcD_0,\mcE_0,\ldots,\mcD_r,\mcE_r)$ where $(\mcD_0,\mcE_0)$
    has a marked vertex corresponding to $s_0$.
    The generating series for staircase diagrams corresponding to pairs
    $(\mcD_k,\mcE_k)$ over $\Gamma_n$ is
    \begin{equation*}
        \displaystyle A_*(t):=A_F(t)\cdot \frac{t}{1-t}.
    \end{equation*}
    Thus the generating function for the number non-fully support staircase
    diagrams over $\tilde\Gamma_n$ is
    \begin{equation*}
        \frac{t\frac{d}{dt}\left(A_*(t)\right)}{1-A_*(t)}+\frac{t^2}{1-t}
    \end{equation*}
    where the second summand corresponds to the generating function of empty
    staircase diagrams.  This completes the proof.
\end{proof}

\begin{proof}[Proof of Theorem \ref{T:enum}]
    Combine Propositions \ref{P:increasing}, \ref{P:broken_increasing},
    \ref{P:finite_full_support}, and \ref{P:total_gen_series}, along with
    Corollary \ref{C:affine_full_support}.
\end{proof}

\bibliographystyle{amsalpha}
\bibliography{palindromic}

\end{document}

%% file: circularstaircase.tex
\def\wedgeoffset{-20} 

\def\wedgelabel[#1,#2,#3,#4,#5,#6,#7]{
    \node at ({(#3+#4)/2 * cos((#1+#2)/2+\wedgeoffset)},#5+#6/2,{(#3+#4)/2 * sin((#1+#2)/2+\wedgeoffset)}) {\textsf{#7}};
}

\def\wedgetop[#1,#2,#3,#4,#5,#6,#7]{
    \draw[smooth,fill=#7] ({#3*cos(#1+\wedgeoffset)},#5,{#3*sin(#1+\wedgeoffset)})
                      -- plot[domain=#1:#2] ({#4*cos(\x+\wedgeoffset)},{#5+#6/(#2-#1)*(\x-#1)},{#4*sin(\x+\wedgeoffset)})
                      -- plot[domain=#1:#2] ({#3*cos(#2+#1-\x+\wedgeoffset)},{#5+#6-#6/(#2-#1)*(\x-#1)},{#3*sin(#2+#1-\x+\wedgeoffset)});
}
\def\wedgeside[#1,#2,#3,#4,#5,#6,#7]{
    \draw[smooth,fill=#7] ({#3*cos(#1+\wedgeoffset)},#5,{#3*sin(#1+\wedgeoffset)})
                      -- plot[domain=#1:#2] ({#3*cos(\x+\wedgeoffset)},{#4+#6/(#2-#1)*(\x-#1)},{#3*sin(\x+\wedgeoffset)})
                      -- plot[domain=#1:#2] ({#3*cos(#2+#1-\x+\wedgeoffset)},{#5+#6-#6/(#2-#1)*(\x-#1)},{#3*sin(#2+#1-\x+\wedgeoffset)});
}
\def\wedgeinnerside[#1,#2,#3,#4,#5,#6]{
    \draw[fill=#6] ({#3*cos(#1+\wedgeoffset)},#5,{#3*sin(#1+\wedgeoffset)}) --
                    ({#3*cos(#1+\wedgeoffset)},#4,{#3*sin(#1+\wedgeoffset)}) --
                    ({#2*cos(#1+\wedgeoffset)},#4,{#2*sin(#1+\wedgeoffset)}) --
                    ({#2*cos(#1+\wedgeoffset)},#5,{#2*sin(#1+\wedgeoffset)}) -- cycle;
}

\def\slantlowerbox[#1,#2,#3,#4,#5,#6,#7,#8,#9]{
    \wedgeside[#1,#2,#4,#5,#6,#7,#8]
    \wedgetop[#1,#2,#3,#4,#6,#7,#8]
    \ifnum #2 < 90  \wedgeinnerside[#2,#3,#4,#5+#7,#6+#7,#8]; \fi
    \ifnum #1 > 90  \wedgeinnerside[#1,#3,#4,#5,#6,#8]; \fi
    \wedgelabel[#1,#2,#3,#4,#6,#7,#9]
}
\def\lowerbox[#1,#2,#3,#4,#5,#6,#7,#8]{
    \slantlowerbox[#1,#2,#3,#4,#5,#6,0,#7,#8]
}

\def\slantupperbox[#1,#2,#3,#4,#5,#6,#7,#8,#9]{
    \wedgeside[#1,#2,#3,#5,#6,#7,#8]
    \wedgetop[#1,#2,#3,#4,#6,#7,#8]
    \ifnum #1 < 270  \wedgeinnerside[#1,#3,#4,#5,#6,#8]; \fi
    \ifnum #2 > 270  \wedgeinnerside[#2,#3,#4,#5+#7,#6+#7,#8]; \fi
    \wedgelabel[#1,#2,#3,#4,#6,#7,#9]
}

\def\upperbox[#1,#2,#3,#4,#5,#6,#7,#8]{
    \slantupperbox[#1,#2,#3,#4,#5,#6,0,#7,#8]
}